\documentclass{amsart}
\usepackage{amsfonts,amssymb,amsmath,amsthm}
\usepackage{url}
\usepackage{enumerate}

\newtheorem{theorem}{Theorem}[section]
\newtheorem{lemma}[theorem]{Lemma}

\newtheorem{proposition}[theorem]{Proposition}

\newtheorem{remark}[theorem]{Remark}

\numberwithin{equation}{section}

\author{Fernanda Botelho}
\address{
Department of Mathematical Sciences\\ 
The University of Memphis\\
 Memphis, TN 38152, USA}
\email{mbotelho@memphis.edu}

\keywords{Zygmund spaces, the little Zygmund space; Hermitian operators; surjective linear isometries; generators of one-parameter groups of surjective isometries. }
\subjclass{Primary: 46E15, Secondary: 47B15, 47B38}
\begin{document}

\title[ Operators on  Zygmund spaces]{ Surjective Isometries and Hermitian Operators on  Zygmund spaces}

\begin{abstract}
In this paper we show that  surjective linear isometries on the little Zygmund space are integral operators and the bounded hermitian operators are trivial.
\end{abstract}
\maketitle

\newpage
\section{Introduction}
The Zygmund space $\mathcal{Z}$ is the set of  all analytic functions $f$ on the open disc $\triangle$,  which are continuously extended to the boundary and satisfy the boundedness condition
\[ \sup_{|z|<1} \, (1-|z|^2) |f''(z)| < \infty.\]
This space endowed with the norm  $\|f\|_{\mathcal{Z}} = |f(0)|+|f'(0)|+\sup_{|z|<1} \, (1-|z|^2) |f''(z)|$ is a Banach space. We recall that the little Zygmund space  is the  closed subspace of $\mathcal{Z}$ defined by (see \cite{zy}):
\[ \mathcal{Z}_0 =\{ f \in \mathcal{Z}: \,  \lim_{|z|\rightarrow 1^-} \, (1-|z|^2) |f''(z)|=0\}.\]
Furthermore, we also consider the subspaces of the little Zygmund space
\[ \mathcal{Z}_0^{(0,1)} =\{ f \in \mathcal{Z}_0: \, f(0)=f'(0)=0 \} \] and
 \[ \mathcal{Z}_0^{i} =\{ f \in \mathcal{Z}_0: \, f^{(i)}(0)=0 \} \,\,\,\mbox{with} \,\,i=0, 1, \, \,\,\, f^{(0)}=f \,\, \mbox{and} \,\, f^{(1)}=f'.\]

Recently, there have been numerous papers on various aspects of classes of operators on Zygmund spaces, see \cite{co_li} and references therein.
In this paper we characterize the surjective isometries supported by these spaces and also classes of operators  that are intrinsically related to the surjective isometries.

In section 2,  we describe the surjective linear isometries supported by $\mathcal{Z}_0^{(0,1)}.$ We start by defining an  embedding of $\mathcal{Z}_0^{(0,1)}$ into a space of continuous functions $\mathcal{C}_0(\triangle)$, then using   the form of the extreme points of the unit ball  of the dual space $\mathcal{C}_0(\triangle)^*$ we give a characterization for the extreme points of $(\mathcal{Z}_0^{(0,1)})^*_1,$ see \cite{bo_fl_ja}.

The adjoint operator of a surjective linear isometry on a Banach space $X$ determines a natural bijection on the set of extreme points of $X^*_1.$ Hence  the action of the adjoint operator  on the set of extreme points  often gives  a representation for the isometries on $X$. This was the approach followed by deLeeuw, Rudin and Werner in the characterization of the surjective isometries on spaces of continuous functions, cf. Theorem 2.3.16 in  \cite{fj}.  We follow this path in our derivation of the form for the surjective isometries supported by  $\mathcal{Z}_0^{(0,1)}$. We show that isometries of $\mathcal{Z}_0^{(0,1)}$ are integral operators of translated weighted differential operators. The form of the isometries encountered in this new setting is quite different from the standard weighted composition operators type of isometries supported by several spaces of analytic functions, as pointed out in \cite{ye-li}, see also \cite{fj,hj} and \cite{hof}.

In section 3, we  use our characterization of the isometries to describe the generators of strongly continuous one-parameter groups of surjective isometries. Thereby  we derive the form  for the hermitian operators from the representation theorem for isometries. Further, we conclude that bounded hermitian operators are trivial. We then employ a theorem in \cite{flja} to extend our representation for the hermitian operators on $\mathcal{Z}_0^{(0,1)}$ to the little Zygmund space $\mathcal{Z}_0.$

\section{Extreme points of $(\mathcal{Z}_0^{(0,1)})^*_1$}

We   embed   $\mathcal{Z}_0^{(0,1)}$ into  $\mathcal{C}_0(\triangle),$ the space of  all continuous functions $F$ defined on the unit disc and satisfying the boundary condition $\lim_{|z| \rightarrow 1}F(z)=0$. This space is  endowed with the norm $\|F\|_{\infty} = \max |F(z)|.$  We define
\begin{equation*} \begin{array}{lccl}  \Phi:&\mathcal{Z}_0^{(0,1)} & \rightarrow & \mathcal{C}_0(\triangle) \\
                      & f  &\rightarrow &  F=\Phi (f):  \triangle \rightarrow  E,\end{array}\end{equation*}  by  $\Phi (f)(z)=(1-|z|^2) f''(z).$
 The map $\Phi$  is a linear isometry with  range space denoted  by $\mathcal{Y}$.

 Throughout this section we represent functions in $\mathcal{Z}_0^{(0,1)}$ with lower case letters and  their images under $\Phi$  with upper case letters, e.g.  $F=\Phi (f)$.

Arens and Kelley's theorem (see  Corollary 2.3.6 in \cite{fj}) states that  every  extreme point of the unit ball of $\mathcal{Y}^*$ is  of the form $e^{i\alpha} \delta_{z},$ with  $z\in \triangle$ and $\delta_z :\mathcal{Y}\, \rightarrow \, \mathbb{C}$  given by  $\delta_z(F)=F(z).$ This implies that  extreme points of  $(\mathcal{Z}_0^{(0,1)})^*_1$ are of the form $\varphi: \mathcal{Z}_0^{(0,1)} \rightarrow \mathbb{C}$ given by $\varphi (f)(z)= e^{i \alpha} (1-|z|^2) f''(z)$.

 We denote the set of extreme points of the unit ball of the dual space $\mathcal{Y}$ by $ext(\mathcal{Y}_1^*)$ and in the next lemma  we  show that every functional of the form $e^{i\alpha} \delta_{z}$ is an extreme point of $\mathcal{Y}^*_1$.

\begin{lemma}\label{Lemma_extreme_points}
 $ext(\mathcal{Y}^*_1)=\{ e^{i\theta}\delta_z : \, z \in \triangle\,\,\theta \in \mathbb{R}\}.$
\end{lemma}
\begin{proof}
Arens and Kelley's theorem  states   that $ext(\mathcal{Y}^*_1)\,\subseteq \, \{ e^{i \theta} \delta_z : \, z \in \triangle\}.$
We now prove the reverse inclusion. Given a functional of the form $e^{i\alpha}\delta_z$, we assume that there exist $\varphi_1$ and $\varphi_2$ in $\mathcal{Y}^*_1$,  such that
\begin{equation} \label{extreme_points}  \delta_{z}= \frac{\varphi_1+ \varphi_2}{2}.\end{equation}
Since $\mathcal{Y}$ is a closed subspace of $\mathcal{C}_0(\triangle)$, the Hahn-Banach Theorem implies the existence of norm 1 extensions of $\varphi_0$ and $\varphi_1$, to  $\mathcal{C}_0(\triangle),$ denoted by $\tilde{\varphi_0}$ and $\tilde{\varphi_1}$ respectively. These functionals are written as
\[ \tilde{\varphi_1}(F) = \int_{\triangle} \, F d \nu \,\,\,\, \mbox{and} \,\,\, \tilde{\varphi_2}(F) = \int_{\triangle} \, F d \mu ,\]
with  $\nu$ and $\mu$ representing  regular probability Borel measures on $\triangle$.

Given $z_0 \in \triangle \setminus \{0\}$, we consider the following function
\[ f_0(z)= (1-|z_0|^2) \left( - \frac{1}{\overline{z_0}} \right)\left[z+\frac{1}{\overline{z_0}} \mbox{Log} (1-\overline{z_0} \, z)\right].\] It is easy to check that $f_0 \in \mathcal{Z}_0^{(0,1)}$. Furthermore $\|f_0\|_{\mathcal{Z}} = |F_0(z_0)| > |F_0(z)|, $ where $F_0(z)= (1-|z|^2) f_0''(z)$ for all $z \neq z_0.$
We apply (\ref{extreme_points}) to the function $F_0$  to conclude that $\tilde{\varphi_0}(F_0)=\tilde{\varphi_1}(F_0)=1.$ If $|\nu|(\triangle \setminus \{z_0\})>0,$ then there exists a compact subset $K$ of $\triangle \setminus \{z_0\}$ such that $|\nu |(K)>0$. Clearly
\[ \sup_{z \in K} |F_0(z) |= \sup_{z \in K} (1-|z|^2)|f_0''(z)|=\alpha <1.\]
Hence
\begin{align*}
1=\varphi_1(F_0)&=|\int_{\triangle} F_0 d \nu| =\left |\int_{\{z_0\}} F_0  d \nu+ \int_K F_0  d \nu+\int_{(\triangle\setminus \{z_0\})\setminus K} F_0  d \nu\right|\\
& \leq |\nu|(\{z_0\}) +\alpha |\nu|(K) +|\nu|((\triangle\setminus \{z_0\})\setminus K) \\
& <|\nu|(\triangle) =1.
\end{align*}
This leads to an absurd and shows that $|\nu| (\triangle \setminus \{z_0\})=0$ and $\nu (\triangle \setminus \{z_0\})=0$. Therefore $\nu (\{z_0\})=1.$  A similar reasoning applies to $\mu $.
Given $F \in \mathcal{Y}$,  we have
\begin{align*}
 \delta_{z_0} (F) = &(1-|z_0|^2) f''(z_0)=\frac{\varphi_0(F)+\varphi_1(F)}{2} \\
 =& \frac{1}{2} \left( \int_{\{z_0\}} F d\nu + \int_{\{z_0\}} F d\mu \right)\\
 =& \frac{1}{2} \left[ \nu (z_0) (1-|z_0|^2) f''(z_0)+ \mu (z_0) (1-|z_0|^2) f''(z_0)\right].
\end{align*}  Therefore
\[  f''(z_0)= \frac{\nu (z_0) (f''(z_0))+\mu(z_0) (f''(z_0))}{2}.\] Then  $\nu =\mu$ and $\varphi_0=\varphi_1.$ This  completes the proof.
\end{proof}
Lemma \ref{Lemma_extreme_points} implies that the extreme points of $(\mathcal{Z}_0^0)^*_1$ are precisely the functionals  $ \Upsilon_{z_0}(f)= e^{i\alpha} (1-|z_0|^2)f''(z_0),$ with $\alpha \in \mathbb{R}$ and $z_0 \in \triangle$.
\begin{remark} \label{zygmund_subs}
 We   observe that $(\mathcal{Z}_0^i)^*= (\mathbb{C}\oplus_1 \mathcal{Z}_0^{(0,1)})^*= \mathbb{C}\oplus_{\infty} ({\mathcal{Z}_0^{(0,1)}})^*$  ($i=0,1$) and also \[(\mathcal{Z}_0)^*= (\mathbb{C}\oplus_1\mathbb{C}\oplus_1 \mathcal{Z}_0^{(0,1)})^*= \mathbb{C}\oplus_{\infty} \mathbb{C}\oplus_{\infty} ({\mathcal{Z}_0^{(0,1)}})^*. \]
It follows  that $ext((\mathcal{Z}_0^{i})^*_1)$ consists of functionals $\tau$ given by
\[ \tau (f) = e^{i \theta_1}f^{(1-i)}(0)z_0^{(1-i)}+e^{i \theta_2}(1-|z_0|^2)f''(z_0),\, \mbox{with} \, z_0 \in \triangle, \, \theta_{1,2} \in [0,\,2\pi).\]
Moreover $ext((\mathcal{Z}_0)^*)_1$ is the set of all functionals  \[\tau (f)= e^{i \theta_0} f(0)+ e^{i \theta_1}f'(0) z_0 + e^{i \theta_2}(1-|z_0|^2) f''(z_0),\] with $\theta_k \in [0,\,2\pi)$ ($k=0,1,2$).
\end{remark}
\section{Characterization of the surjective isometries on $\mathcal{Z}_0^{(0,1)}$}
In this section we show that  linear surjective isometries on $\mathcal{Z}_0^{(0,1)}$ can be represented as integral operators.

Given a surjective linear isometry $T: \mathcal{Z}_0^{(0,1)} \rightarrow \mathcal{Z}_0^{(0,1)}$ we denote by  $S: \mathcal{Y} \rightarrow \mathcal{Y}$ the corresponding isometry on $\mathcal{Y}$ such that   $S = \Phi \circ T \circ \Phi^{-1},$ where $\Phi$ represents the embedding considered in the previous section.  The adjoint operator of $S$, $S^* : \mathcal{Y}^* \rightarrow \mathcal{Y}^*$ induces a permutation on the set of extreme points  of  $(\mathcal{Y}^*)_1$. This can be  expressed as follows.
For every  $z \in \triangle$ and $\theta$ there exists a unique pair $(w, \alpha)\in \triangle\times [0, \, 2\pi)$   such that
\[ S^* ( e^{i\theta} \delta_{z} ) =  e^{i \alpha} \delta_{w}.\]
Equivalently
\begin{equation} \label{meq}
(1-|z|^2) e^{i \theta} (Tf)''(z)= (1-|w|^2) e^{i \alpha} (f''(w)), \,\,\, \mbox{for every } f \in \mathcal{Z}_0^{(0,1)}.
\end{equation}

The values of $\alpha$ and $w$ conceivably  depend on the choice of $\theta$ and $z$.  This determines the following  two maps:
\begin{equation} \label{sigma_Gamma} \begin{array}{rlll}\sigma_0 :& \mathbb{S}_1 \times \triangle & \rightarrow &\triangle  \\ & (e^{i\theta}, z)& \rightarrow & w \end{array} \,\,\,\, \mbox{and} \,\,\, \begin{array}{rlll}\Gamma_0 :& \mathbb{S}_1 \times \triangle  & \rightarrow & \mathbb{S}_1 \\ & (e^{i\theta} ,z)& \rightarrow &  e^{i\alpha}. \end{array}\end{equation}
 We show in the following lemma that $\sigma_0$ is independent of the first coordinate and then we consider $\sigma: \triangle \rightarrow \triangle $ given by $\sigma (z)= \sigma_0 (1, z).$

\begin{lemma} \label{sigma} If $z \in \triangle$, then  $\sigma_0$ restricted to the set  $\{ (e^{i\theta}, z): \theta \in \mathbb{R}\}$ is constant and $\sigma:\triangle \rightarrow \triangle,$ defined by $\sigma(z) = \sigma_0 (1,z),$ is a disc automorphism.
\end{lemma}
\begin{proof}
We assume that there are points in  $\mathbb{S}_1$, $e^{i\theta}$ and $e^{i\theta_1}$ such that $\sigma_0 (e^{i\theta}, z)=w \neq w_1 =\sigma_0 (e^{i\theta_1}, z)$, for some value of $z \in \triangle$. Hence
\begin{equation} \label{2_eq_a}
 (1-|z|^2) e^{i \theta} ((Tf)''(z)) = (1-|w|^2) e^{i \alpha}(f''(w))\end{equation}
and
 \begin{equation} \label{2_eq_b}(1-|z|^2) e^{i \theta_1} ((Tf)''(z)) = (1-|w_1|^2) e^{i \alpha_1}(f''(w_1)). \end{equation}
Substituting   $f_0(z)= z^2/2$ into (\ref{2_eq_a}) and (\ref{2_eq_b}) we get
\[ (1-|z|^2) e^{i \theta}((Tf_0)''(z))=(1-|w|^2) e^{i \alpha} \,\,\,\mbox{and} \,\,\, (1-|z|^2) e^{i \theta_1} ((Tf_0)''(z))=(1-|w_1|^2) e^{i \alpha_1},\]  respectively.

Therefore $ e^{i (\alpha-\theta)} (1-| w|^2)= e^{i (\alpha_1-\theta_1)} (1-| w_1|^2)$. This implies that  $|w|=|w_1|$ and  $e^{i (\alpha-\theta)}=e^{i(\alpha_1- \theta_1) }.$ From (\ref{2_eq_a}) and (\ref{2_eq_b}) we conclude that  $f''(w)=f''(w_1),$ for every $f \in \mathcal{Z}^{(0,1)}_0.$ Hence $w=w_1$ and $\sigma_0$ depends only on the value of $z$, as claimed.

Then given  $\sigma$ as in the statement of the lemma, we write  (\ref{2_eq_a}) as
 \begin{equation} \label{n_eq_a}
 (1-|z|^2) e^{i \theta}  (Tf)''(z)  = (1-|\sigma (z)|^2) e^{i \alpha} f''(\sigma (z) ).\end{equation}

 We apply the same reasoning to  $T^{-1}$ to determine $\psi,$ a mapping from the open disc into itself, satisfying the equation
 \begin{equation} \label{n_eq_for_inv}
 (1-|z|^2) e^{i \theta}  (T^{-1}f)''(z)  = (1-|\psi (z)|^2) e^{i \beta} f''(\psi (z) ).\end{equation}
Equation (\ref{n_eq_for_inv}) applied to $Tf$ yields
\[ \begin{array}{rl} (1-|z|^2) e^{i \theta} f''(z)& = (1-|\psi (z)|^2) e^{i \beta}(Tf)''(\psi (z) )\\  & \\ & =e^{i (\beta - \theta + \alpha)} (1-|\sigma(\psi(z))|^2) f''(\sigma(\psi(z))), \end {array} \]
then
\[  (1-|z|^2) e^{i \theta} f''(z)=e^{i (\beta - \theta + \alpha)} (1-|\sigma(\psi(z))|^2) f''(\sigma(\psi(z))), \,\,\, \mbox{for every } \,\,f \in  \mathcal{Z}_0^{(0,1)} \,\, z \in \triangle .\]
Setting  $f(z)=z^2/2$ in the equation displayed above, we obtain $$(1-|z|^2) e^{i \theta}=e^{i (\beta - \theta + \alpha)} (1-|\sigma(\psi(z))|^2). $$ This implies  $|z|=|\sigma(\psi(z))|$ and  $e^{i \theta}=e^{i (\beta - \theta + \alpha)}.$ Therefore $f''(z)=f''(\sigma(\psi(z)))$ which implies

that  $\sigma \circ \psi $ is the identity on $\triangle$ and then $\sigma$ is surjective. A similar reasoning also shows that $\psi \circ \sigma $ is the identity on $\triangle$ and $\sigma$ is injective.
We now  prove that $\sigma$ is analytic. To this end,  we apply the equation (\ref{n_eq_a}) to the two following functions $f_0(z)= \frac{z^2}{2} $ and $f_1(z)= z^3/6.$ We obtain
\[ (1-|z|^2) e^{i \theta} [(Tf_0)''(z)]= (1-|\sigma(z)|^2) e^{i \alpha} \]
and
\[ (1-|z|^2)e^{i \theta} [(Tf_1)''(z)]= (1-|\sigma(z)|^2)e^{i \alpha} \sigma(z) ,\] respectively.
For every $z \in \triangle $ we have $[(Tf_0)''(z)]\neq 0.$ Therefore
\[ \sigma(z) = \frac{[(Tf_1)''(z)]}{[(Tf_0)''(z)]}.\] This shows that $\sigma $ is analytic and then a disc automorphism.
\end{proof}

\begin{theorem} \label{mt} Let  $T: \mathcal{Z}_0^{(0,1)} \rightarrow \mathcal{Z}_0^{(0,1)}$, then $T$ is surjective linear isometry if and only if there exist a  disc automorphism $\sigma $ and a real number $\alpha$ such that for every $f \in  \mathcal{Z}_0^{(0,1)} $ and $z \in \triangle,$
\[ Tf(z)= e^{i\alpha} \,\int_0^z \, [f'\circ\sigma)(z)-(f'\circ\sigma)(0)]dz.\]
\end{theorem}
\begin{proof}
We first assume that  $T$ is a surjective linear isometry. It follows from Lemma \ref{sigma} and  respective preamble that
\begin{equation} \label{temp_eq} (1-|z|^2) e^{i \theta} (Tf)''(z)= (1-|\sigma (z)|^2) \Gamma_0 (\theta, z) (f''(\sigma (z ))),  \end{equation}
for every $f \in \mathcal{Z}_0^{(0,1)}$ and $z \in \triangle$.

In particular, for $f_0(z)=z^2/2$, we have
\[ \left| \frac{ (Tf_0)''(z)}{\sigma '(z)} \right| =1.\] Since the mapping $z \rightarrow \frac{ (Tf_0)''(z)}{\sigma '(z)}$ is analytic on the open disc, the Maximum Modulus Principle for analytic functions implies that it must be constant, then there exists $\eta \in [0, \, 2\pi)$ such that
\[ \frac{ (Tf_0)''(z)}{\sigma '(z)}=e^{i \eta}.\]
The equation displayed in  (\ref{temp_eq}) applied to $f_0$ yields
$$e^{i \theta} e^{i \eta}= \frac{|\sigma'(z)|}{\sigma'(z)} \Gamma_0 (\theta, z) .$$
 Substituting this relation in  (\ref{temp_eq}), we have
\[ (Tf)'' (z) = e^{i \eta} \sigma'(z) f''(\sigma (z)).\]
Integrating this last equation twice and since $(Tf)'(0)= Tf (0)=0$, we obtain
\[ (Tf)(z) = e^{i \eta} \int_0^z \left[ f'(\sigma (\xi)) - f'(\sigma (0))\right] d \xi.\]
It is easy to check that $T$ of the form displayed in the statement of the theorem is an isometry. To this end and since $(1-|z|^2) |\sigma'(z)|= |\sigma(z)|,$ we have
\[\begin{array}{rl} sup_{|z|<1} (1-|z|^2) |(Tf)''(z)| & = sup_{|z|<1} (1-|z|^2) |f'' (\sigma (z)) \sigma'(z)| \\
& \\ & = \sup_{|z|<1} (1-|z|^2) |f''  (z)|.\end{array} \]
 This completes the proof.

\end{proof}

\section{ Strongly continuous one parameter groups of surjective isometries on $\mathcal{Z}_0^{(0,1)}$}
Let $\{T_t\}_{t \in \mathbb{R}}$ be a one parameter group of surjective isometries on $\mathcal{Z}_0^{(0,1)}$. For each $t$, $T_t$ has the representation
\[ T_t(f) (z) = e^{i\alpha_t} \,\int_0^z \, [f'\circ\sigma_t)(\xi )-(f'\circ\sigma_t)(0)]d\xi .\]
In this section we show that the group properties of $\{T_t\}_{t \in \mathbb{R}}$ transfer to the families  $\{\alpha_t\}_{t \in \mathbb{R}}$ and also $\{\sigma_t\}_{t \in \mathbb{R}}$ defining $\{T_t\}_{t \in \mathbb{R}}$.

We recall that $\{T_t\}_{t \in \mathbb{R}}$ being a strongly continuous one-parameter group means that  $T_0 = Id$, $T_{s+t} = T_s T_t$ for every $s, t \in \mathbb{R} $ and the strong continuity means that  \[lim_{t \rightarrow 0} \sup_{|z|<1} (1-|z|^2) |(T_tf)''(z)-f''(z)|=0,\]
for every $f \in \mathcal{Z}_0^{(0,1)}.$
\begin{proposition}\label{group_properties} Let $\{T_t\}_{t \in \mathbb{R}}$ be a family of surjective linear isometries on $\mathcal{Z}_0^{(0,1)}$. Then
 $\{T_t\}_{t\in \mathbb{R}}$ is a strongly continuous one parameter group  if and only if there exist a continuous one parameter group of disc automorphisms $\{\sigma_t\}_{t\in \mathbb{R}}$ and a complex number $\alpha$  such that
 \[ T_t (f)(z)= e^{i \alpha t} \int_0^z [ f'(\sigma_t(\xi))-f(\sigma_t(0))] d \xi, \,\,\, \forall \,\, f \in \mathcal{Z}_0^{(0,1)}.\]
\end{proposition}
\begin{proof}
Since  $T_0 = Id$, we have that  $e^{i \alpha_0} \int_0^z [ f'(\sigma_0(\xi))-f(\sigma_0(0))] d \xi=f(z),$ for every $f$.

This implies that
$e^{i \alpha_0} \left[ f'(\sigma_0(z))-f'(\sigma_0(0)) \right] =f'(z)$ and $ e^{i \alpha_0} f''(\sigma_0 (z)) =f''(z).$ Applying this equation to $f_0(z)=z^2/2$ and to $f_1(z)=z^3/6$ we get $e^{i \alpha_0}=1$ and $\sigma_0=id.$

Since $T_s T_t =T_{s+t}$ we have that  $$ \begin{array}{rl} T_s [T_t (f) ](z) & = e^{ i \alpha_t } \int_0^z \left[ (T_s (f))'(\sigma_t (\xi)) - (T_s f)'(\sigma_t (0))\right] d\xi \\
& \\
& = e^{ i \alpha_t } \int_0^z \left[ e^{i \alpha_s} [ f'(\sigma_s (\sigma_t(\xi)) - f'(\sigma_s (\sigma_t(0)) ] \right] d\xi \\
& \\
& = e^{ i (\alpha_t + \alpha_s) } \int_0^z \left[ f'(\sigma_s (\sigma_t(\xi)) - f'(\sigma_s (\sigma_t(0)) \right] d\xi\\
& \\
& = e^{ i \alpha_{t+s}}  \int_0^z \left[ f'(\sigma_{s +t} (\xi)) -f'(\sigma_{s+t} (0)) \right] d\xi. \end{array} $$

Differentiating this last equation we have
\[e^{ i (\alpha_t + \alpha_s) }  \left[ f'(\sigma_s (\sigma_t(z)) - f'(\sigma_s (\sigma_t(0)) \right] =e^{ i \alpha_{t+s}}  \left[ f'(\sigma_{s +t} (z)) -f'(\sigma_{s+t} (0)) \right]  \]
and differentiating again
\[e^{ i (\alpha_t + \alpha_s) }   f''(\sigma_s (\sigma_t(z))  \sigma_s'(\sigma_t(z)) \tau_t'(z)= e^{ i \alpha_{t+s}}   f''(\sigma_{s +t} (z))  \sigma'_{s+t}(z) ,   \]
for every $f \in \mathcal{Z}_0^{(0,1)}$ and $z \in \triangle . $
In particular for $f(z) = z^2/2$ we have $e^{ i (\alpha_t + \alpha_s) }\sigma_s'(\sigma_t(z)) \sigma_t'(z)= e^{ i \alpha_{t+s}}\sigma'_{s+t}(z).$ Thus $ f''(\sigma_s (\sigma_t(z)))=f''(\sigma_{s +t} (z)) $ holds for every $f \in \mathcal{Z}_0^{(0,1)}$ and $z \in \triangle . $

Also, for $f(z) = z^3/6$, we have $\sigma_s (\sigma_t(z))=\sigma_{s +t} (z)$ and $e^{ i (\alpha_t + \alpha_s) }=e^{ i \alpha_{t+s}}.$
Since $\alpha_t$ is one parameter group of scalars then it must be of the form $\alpha t$.
It is left to prove that that $t \rightarrow \sigma_t$ is continuous at $t=0$. Since $T_t$ is continuous at $t=0$, for every $f$,
\[ \lim_{t \rightarrow 0} \,\sup_{|z|<1} \, (1-|z|^2) | e^{i \alpha t} f'' (\sigma_t(z)) \sigma'(t) - f''(z) |=0 .\]
In particular for $f(z) = z^2/2$ and for $f(z) = z^3/6$ we have
\[ \lim_{t \rightarrow 0} \,\sup_{|z|<1} \, (1-|z|^2) | e^{i \alpha t}  \sigma'(t) - 1 |=0 ,\]
and
\[ \lim_{t \rightarrow 0} \,\sup_{|z|<1} \, (1-|z|^2) | e^{i \alpha t} \sigma_t(z) \, \sigma_t'(z) - z |=0 ,\] respectively.
For every $z \in \triangle$,
\begin{equation} \label{good_ine} \lim_{t \rightarrow 0}| e^{i \alpha t}  \sigma_t'(z) - 1 |=0 \,\,\mbox{and} \,\,\lim_{t \rightarrow 0}| e^{i \alpha t} \sigma_t(z) \, \sigma_t'(z) - z |=0.\end{equation}

Since
\[ |\sigma_t(z)-z| \leq |\sigma_t(z)| | e^{i\alpha t}\sigma_t'(z)-1| + |\sigma_t(z) \sigma_t'(z)-z|\]
the limits displayed in (\ref{good_ine}) imply that, for every $z \in \triangle ,$
\begin{equation} \label{pointwise_ct} \lim_{t \rightarrow 0}|   \sigma_t (z) - z |=0.\end{equation}
For each $t \in \mathbb{R}$, $\sigma_t(z) =\lambda_t \frac{z-a_t}{1-\overline{a_t} \, z}$ with $\lambda_t$ a modulus 1 complex number and $a_t \in \triangle$.
For $z=0$, (\ref{pointwise_ct}) implies that $\lim_{t \rightarrow 0} \lambda_t a_t =0$, then  $\lim_{t \rightarrow 0}  a_t =0.$ Suppose $\lambda_{t_n}$ is a convergent sequence of modulus 1 complex numbers, we assume that it  converges to $e^{i \theta}$, from (\ref{pointwise_ct}) we conclude that $e^{i \theta}=1$ and every   sequence $\lambda_{t_n}$  it must converge to $ 1.$ Therefore, there exists $\delta >0$ such that for  $|t|< \delta$ we have $|a_t|< \min \{1/2, \epsilon /3 \}$ and $|\lambda_t -1|< \epsilon /3$. This implies  that
\[ \begin{array}{rl} |\sigma_t(z) -z| & = \frac{(\lambda_t-1) z - \lambda_t a_t + \overline{a_t} z^2|}{|1-\overline{a_t} z|} \\
& \\
& \leq \frac{\left[ |\lambda_t-1|+ 2 |a_t|  \right] }{ 1-|a_t|}< 2 \epsilon , \,\,\mbox{ for every } \,\, z \in \triangle.\end{array} \]
This shows that $\sigma_t$ is uniformly continuous at $t=0$ and  completes the proof.

\end{proof}
\section{Generators of strongly continuous one-parameter groups of isometries on $\mathcal{Z}^{(0,1)}_0$}
We derive the form of the generator of $\{T_t\}_t$, a strongly continuous one parameter group of surjective linear isometries on $\mathcal{Z}^{(0,1)}_0$. We recall that
\[T_t(f)(z)= e^{i\alpha t} \,\int_0^z \, [f'\circ\sigma_t)(z)-(f'\circ\sigma_t)(0)]dz,\]
and its generator is defined as follows:
\[ \mathcal{G}(f)(z)=\left( -i \frac{d}{dt}T_t\right)|_{t=0} f(z ).\]

Therefore
\begin{align*}
\mathcal{G}f(z) & = -i \left[ i \alpha \int_0^z f'(\xi) \, d\xi +   \partial_t \sigma_t (z)|_{t=0} f'(z) \right]\\
&  = \alpha f(z) -i  \partial_t \sigma_t (z)|_{t=0} f'(z).
\end{align*}

Results in \cite{bp} imply that  $\{\sigma_t \}_t$ is either the trivial group or a group of automorphisms   of one of the following types:
\begin{enumerate}
\item[(i)] Elliptic.\[  \sigma_t (z) = \frac{ ( e^{ict}-|\tau|^2)z - \tau (e^{ict}-1)}{ 1-|\tau|^2 e^{ict} - \bar{\tau} (1- e^{ict})z}, \] with $c \in \mathbb{R}\setminus \{0\}$, $ \tau \in \mathbb{C}$ such that $|\tau|<1.$
    \item[(ii)] Hyperbolic.
    \[ \sigma_t (z) = \frac{(\beta e^{\varphi t }-\alpha)z + \alpha \beta (1- e^{\varphi t})}{ ( e^{\varphi t}-1)z + (\beta - \alpha e^{\varphi t})},\]
    with $\varphi$ a positive real number, $|\alpha|=| \beta|=1$ and $\alpha \neq \beta.$
 \item[(iii)] Parabolic.
    \[ \sigma_t(z) = \frac{(1-ict) z + ict \gamma}{-ic \bar{\gamma} t z +1+ i c t},\] with $c \in R \setminus\{0\}$ and $|\gamma|=1.$
\end{enumerate}
Disc automorphisms can be extended to the conformal maps on the plane and as such they fall in one of the three different types listed above according to the fixed points. More precisely,  an  elliptic automorphism has a single fixed in the disc and another one in the interior of  its complement; a hyperbolic automorphism has two distinct fixed points on the boundary of the disc and  a parabolic has a single fixed point on the boundary of the disc. It is shown in \cite{bp} that every automorphism in a 1-parameter group of disc automorphisms share the same fixed points.

Therefore, we summarize these considerations in the next Proposition.
\begin{proposition} \label{hermitians}
An hermitian operator $\mathcal{G}$ on $\mathcal{Z}^{(0,1)}_0$ is of one of the following forms:
\begin{enumerate}
\item  $\mathcal{G}f(z)= \alpha f(z),$ $\alpha \in \mathbb{R}.$
\item \[\mathcal{G}f(z)= \alpha f(z)- c \, \frac{(\overline{\tau} z -1)(z-\tau)}{1 - |\tau|^2} f'(z), \] with $c \in \mathbb{R}\setminus \{0\}$, $ \tau \in \mathbb{C}$ such that $|\tau|<1.$
\item \[ \mathcal{G}f(z)= \alpha f(z)+\frac{i \varphi}{\beta_0- \beta_1} (z -\beta_0) (z-\beta_1 ) f'(z) ,\]
    with $\varphi$ a positive real number, $|\beta_0|=| \beta_1|=1$ and $\beta_0 \neq \beta_1.$
\item  \[ \mathcal{G}f(z)= \alpha f(z)+ c \overline{\gamma} ( z- \gamma)^2 \, f'(z),\] with $c \in R \setminus\{0\}$ and  $|\gamma|=1.$
\end{enumerate}
\end{proposition}
\begin{remark} \label{hermitian_trivial}
Proposition \ref{hermitians} implies that bounded hermitian operators on $\mathcal{Z}^{(0,1)}_0$ are trivial, i.e. $\mathcal{G}f(z)= \alpha f(z)$ with $\alpha \in \mathbb{R}.$ We observe that $f(z)=z^2$  is not in the domain of $\mathcal{G}$ for the elliptic, hyperbolic and parabolic cases.  It also follows from standard computations that the point spectrum of  $\mathcal{G}$ is is a singleton for $\mathcal{G}f(z)= \alpha f(z)$ and empty for the remaining cases.
\end{remark}

We now extend the characterization of the surjective isometries to the Little Zygmund space. To this end, we employ results of Fleming and Jamison in \cite{flja}, namely Theorem 3.7(a) and Theorem 3.3. As noted earlier,  $\mathcal{Z}_0 =\left(\mathbb{C}\oplus_1 \mathbb{C} \right)\oplus_1 \mathcal{Z}_0^{(0,1)}$,  Remark \ref{hermitian_trivial} implies that $\mathcal{Z}^{(0,1)}_0$ supports only trivial hermitian projections. It is  clear that $\mathbb{C}$ supports only trivial hermitian projections as well. From Theorem 3.7(a) in \cite{flja} we have that  $T: \mathcal{Z}_0 \rightarrow \mathcal{Z}_0$ is a surjective linear isometry if and only if
\[ T(f)(z) = e^{1 \theta}f(0)+e^{i\eta} f'(0)z + e^{i\alpha} \int_0^z [ f'(\sigma (\xi))-f'(\sigma(0))] d \xi,\]  with $\theta$, $\eta$, $\alpha$ real numbers and $\sigma$ a disc automorphism.

 Similar, from Theorem 3.3 in \cite{flja}, $S$ is a bounded hermitian operator  on $\mathcal{Z}_0$ if and only if there exist real numbers $\alpha_1$, $\alpha_2$ and $\alpha_3 $ such that
 \[ (Sf)(z)= \alpha_1 f(0) + \alpha_2 f'(0) z  + \alpha_3 f(z),\]
 for all $f \in \mathcal{Z}_0$ and $z \in \triangle.$

 \textbf{Acknowledgement.} I am thankful to Professor J. Jamison for many enlightening discussions.


\begin{thebibliography}{99}
\bibitem{bp} E. Berkson and H. Porta, Hermitian operators and one parameter groups in Hardy spaces, Trans. Amer. Math. Soc. \textbf{185} (1973), 373--391.
\bibitem{bl} O. Blasco, M. Contreras, S. Diaz-Madrigal, J. Martinez and A. Siskakis, Semigroups of composition operators in BMOA and the extension of a theorem of Sarason, Michigan Math. J. \textbf{25} (1978), 101-115.
\bibitem{bl1} O. Blasco, M. Contreras, S. Diaz-Madrigal, J. Martinez, M. Papadimitrakis and A. Siskakis, Semigroups of composition operators and integral operators in spaces of analytic functions, (2010) preprint.
\bibitem{bd} F. F. Bonsall and J. Duncan, Complete normed algebras, Springer-Verlag, New York-Heidelberg, 1973.
\bibitem{bo_fl_ja} F. Botelho, R. Fleming and J. Jamison, Extreme points and isometries on vector-valued Lipschitz spaces, Journal of  Mathematical Analysis and Applications \textbf{381} (2011), 821--832.
\bibitem{bo_ja} F. Botelho and J. Jamison, Generalized bi-circular projections on spaces of analytic functions, Acta Sci. Math. (Szeged)
 \textbf{75} (2009), 527--546.
\bibitem{co_li} F. Colonna and S. Li, Weighted composition operators from $H^{\infty}$ into the Zygmund spaces, Complex Anal. Oper. Theory \textbf{7} (2013), 1495--1512.
\bibitem{ci_wo} J. Cima and W. Wogen, On isometries of the Bloch space, Illinois Journal of Mathematics \textbf{24:2} (1980), 313--316.
\bibitem{en} K-J. Engel and R. Nagel, A short course on operator semigroups, Universitext, Springer, 2006.
\bibitem{fj} R. Fleming and J. Jamison, Isometries on Banach Spaces: Function Spaces, Chapman \& Hall/CRC, Boca Raton, 2003.
\bibitem{flja} Fleming, R.\ J.\ ;Jamison, J. E. {\em Hermitian operators and isometries on sums of Banach spaces}, Proc. Edinburgh Math. Soc. \textbf{32:2} (1989), 169--191.
\bibitem{hof} K. Hoffman, Banach spaces of analytic functions, Dover Publ, Inc., New York, 1962.
\bibitem{hj} W. Hornor and J. E. Jamison, Isometries of some Banach spaces of analytic functions, Integr. Equ. Theory \textbf{41} (2001), 410--425.
    \bibitem{zy} A. Zygmund, trigonometric Series, Volume I \& II (Third Edition) Chambridge University Press, 1935.
\bibitem{ye-li} S. Ye and Q. Hu, Weighted Composition Operators on
the Zygmund Space, Hindawi Publishing Corporation Abstract and Applied Analysis Volume \textbf{2012}, Article ID 462482 (2012), 18 pages.
doi:10.1155/2012/462482
\end{thebibliography}
\end{document}